\documentclass[11pt,hyp, reqno]{nyjm}
\usepackage{hyperref}
\hypersetup{nesting=true,debug=true,naturalnames=true}
\usepackage{graphicx,amssymb,upref}
\usepackage[active]{srcltx}
\usepackage[all]{xy}
\usepackage{amsthm}
\usepackage{amsfonts}
\usepackage{amsmath}
\usepackage{amssymb}
\usepackage{mathrsfs}
\usepackage[active]{srcltx}
\usepackage[all]{xy}
\usepackage{url}
\usepackage{mathtools}

\newcommand{\doi}[1]{doi:\,\href{http://dx.doi.org/#1}{#1}}

\makeatletter
\@namedef{subjclassname@2020}{%
  \textup{2020} Mathematics Subject Classification}
\makeatother

\newcommand{\ncom}{\newcommand}     
\newcommand*{\Le}{\leqslant}
\newcommand*{\Ge}{\geqslant}
\ncom{\beqn}{\begin{eqnarray*}}
\ncom{\eeqn}{\end{eqnarray*}}
\ncom{\beq}{\begin{eqnarray}}
\ncom{\eeq}{\end{eqnarray}}
\ncom{\inp}[2]{\langle{#1},\,{#2} \rangle}
\ncom{\rar}{\rightarrow}

\newtheorem{theorem}{Theorem}[section]

\newtheorem{lemma}[theorem]{Lemma}
\newtheorem{proposition}[theorem]{Proposition}
\newtheorem{corollary}[theorem]{Corollary}
\theoremstyle{definition}

\theoremstyle{remark}
\newtheorem{remark}[theorem]{Remark}
\newtheorem{question}{Question}[section]

\newtheorem{example}[theorem]{Example}

\title[Multipliers of the Hilbert spaces of Dirichlet series]
{Multipliers of the Hilbert spaces of Dirichlet series}
   \author{Chaman Kumar Sahu}
   \address{Department of Mathematics and Statistics\\
Indian Institute of Technology Kanpur, India}
   \email{chamanks@iitk.ac.in}


\begin{document} 
 
\begin{abstract}  
	For a sequence  $\mathbf w = \{w_j\}_{j = 2}^\infty$ of positive real numbers,   consider the positive semi-definite kernel $\kappa_\mathbf w(s, u) = \sum_{j = 2}^\infty w_j j^{-s - \overline{u}}$ defined on some right-half plane $\mathbb H_{\rho}$ for a real number $\rho.$ Let $\mathscr H_\mathbf w$ denote the reproducing kernel Hilbert space associated with $\kappa_\mathbf w.$ 
Let 
\beqn
\delta_\mathbf w = \inf\Bigg\{\Re(s) : \sum\limits_{\substack{j \Ge 2 \\\tiny{\mbox{\bf gpf}}(j) \Le p_n }} w_j j^{- s} < \infty ~\text{for all}~ n \in \mathbb Z_+\Bigg\},  
\eeqn
where $\{p_j\}_{j \Ge 1}$ is an increasing enumeration of prime numbers and $\text{\bf gpf}(n)$ denotes the greatest prime factor of an integer $n \Ge 2.$ 
If $\mathbf w$ satisfies 
\beqn
\sum_{\substack{j \Ge 2\\ j | n}} j^{-\delta_\mathbf w} w_j \mu\Big(\frac{n}{j}\Big) \Ge 0,\quad n \Ge 2,
\eeqn
where $\mu$  is the M$\ddot{\mbox{o}}$bius function, then the multiplier algebra $\mathcal M(\mathscr H_\mathbf w)$ of $\mathscr H_\mathbf w$ is isometrically isomorphic to the space of all bounded and holomorphic functions on $\mathbb H_\frac{\delta_{\mathbf w}}{2}$ that are representable by a convergent Dirichlet series in some right half plane.
As a consequence, we describe the multiplier algebra $\mathcal M(\mathscr H_\mathbf w)$ when $\mathbf w$ is an additive function satisfying $\delta_{\mathbf w} \Le 0$ and
\begin{align*}
	\frac{w_{p^{j-1}}}{w_{p^j}} \Le  p^{-\delta_{\mathbf w}}~\text{for all integers} ~~ j \Ge 2~\mbox{and all prime numbers}~p.
\end{align*}
Moreover, we recover a result of Stetler that classifies the multipliers of $\mathscr H_\mathbf w$ when $\mathbf w$ is multiplicative. 
The proof of the main result is a refinement of the techniques of Stetler.
\end{abstract} 
\maketitle
\tableofcontents


 \section{Preliminaries}

Let $\mathbb Z_+$ and $\mathbb N$ denote the sets of positive and non-negative integers, respectively. For $k \in \mathbb N,$ let $\mathbb N_k = \{m \in \mathbb N : m \Le k\}.$
Denote by $\mathbb R$ and $\mathbb C,$ the sets of real and complex numbers, respectively. 
For $s \in \mathbb C,$ let $\Re(s),$ $|s|,$ $\overline{s}$ and $\arg(s)$ denote the real part, the modulus, the complex conjugate and the argument of $s,$ respectively. The open unit disc $\{z \in \mathbb C : |z| < 1\}$ is denoted by $\mathbb D.$ 
For $\rho \in \mathbb R,$ let $\mathbb H_\rho$ denote the right half-plane $\{s \in \mathbb C : \Re(s) > \rho\}.$

We invoke here some known arithmetic functions on $\mathbb Z_+$ used in the sequel. The {\it prime omega function} $\omega$ that counts the total number of all distinct prime factors of a positive integer. The {\it divisor function} $\mathbf d$ counts all divisors of a positive integer. The {\it M$\ddot{o}$bius function} $\mu: \mathbb Z_+ \rar \mathbb C$ is given by
\beq
\label{mu-fun}
\mu(n) = 
\begin{cases}
	1 & \text{if n = 1},\\ 
	(-1)^j & \text{if $n$ is a product of $j$ distinct primes},\\
	0  & \text{otherwise}.
\end{cases}
\eeq

A {\it Dirichlet series} is a series of the form
\beqn
f(s) = \sum_{j=1}^{\infty} a_{j} j^{-s},
\eeqn
where $a_j \in \mathbb C.$  
If $a_j = 1$ for all $j \Ge 1,$ then we have the Riemann zeta function, denoted by $\zeta.$   
If $f$ is convergent at $s=s_{0}$, then it converges 
uniformly throughout the angular region $\{s \in \mathbb C : |\arg(s-s_0)| < \frac{\pi}{2} -\delta\}$ for every positive real number $\delta < \frac{\pi}{2}.$ Consequently, $f$ defines a holomorphic function on $\mathbb H_{\Re(s_0)}$ (refer to \cite[Chapter~IX]{Ti} for the basic theory of Dirichlet series). Let $\mathcal D$ denotes the set of all functions which are representable by a convergent Dirichlet series in some right half plane. 	 

The following proposition describes the product of two Dirichlet series (see \cite[Theorem~4.3.1 and discussion prior to Theorem~4.3.4]{QQ}). 

\begin{proposition}\label{prod-series}
	Let $f(s) = \sum_{j=1}^\infty a_j j^{-s}$ and $g(s) = \sum_{j=1}^\infty b_j j^{-s}$ be two convergent Dirichlet series on $\mathbb H_\rho.$ If $g$ converges absolutely on $\mathbb H_\rho,$ then 
	\beqn
	f(s)g(s) = \sum_{j=1}^\infty \Big(\sum_{m | j} a_m b_{\frac{j}{m}}\Big) j^{-s},
	\eeqn
	which converges on $\mathbb H_\rho.$
\end{proposition}

For an integer $n \Ge 2,$ let $\text{\bf gpf}(n)$ denote the greatest prime factor of $n.$ For a sequence $\mathbf x = \{x_j\}_{j \Ge 2}$ of non-negative real numbers, let 
\beqn
\sigma_{\mathbf x} := \inf\Big\{\Re(s) :  \sum_{j=2}^\infty x_j j^{-s} < \infty \Big\}, 
\eeqn
\beq \label{delta-form}
\delta_\mathbf x := \inf\Bigg\{\Re(s) : \sum\limits_{\substack{j \Ge 2 \\\tiny{\mbox{\bf gpf}}(j) \Le p_n }} x_j j^{- s} < \infty ~\text{for all}~ n \in \mathbb Z_+\Bigg\},  
\eeq
where $\{p_j\}_{j=1}^\infty$ is the increasing enumeration of prime numbers. Note that $\delta_\mathbf x \Le \sigma_\mathbf x.$

\subsection{Multipliers of a reproducing kernel Hilbert space}
Let $X$ be a non-empty open subset of $\mathbb C$ and 
let $H^{\infty}(X)$ denote the Banach space of bounded holomorphic functions on $X$ endowed with supremum norm. 
Let $\mathcal H$ be a reproducing kernel Hilbert space of complex-valued holomorphic functions on $X.$ A function $\varphi :X \rar \mathbb C$ is called {\it multiplier of $\mathcal H$} if $\varphi \cdot f \in \mathcal H$ for every $f \in \mathcal H.$ Clearly, $\mathcal M(\mathcal H)$ is an algebra. Denote by $\mathcal M(\mathcal H),$ the set of all multipliers of $\mathcal H.$ Note that if the constant function equal to $1$ belongs to $\mathcal H,$ then by \cite[Corollary~5.22]{P-R}, 
$\mathcal M(\mathcal H)$ is contractively contained in $H^{\infty}(X).$

For a sequence $\mathbf w = \{w_{j}\}_{j = 2}^{\infty}$ of positive real numbers,  
consider the weighted Hilbert space of the formal Dirichlet series 
\beqn \mathscr H_{\mathbf w} : = \Bigg\{f(s)= \sum_{j = 2}^{\infty} \hat{f}(j) j^{-s}   : \|f\|^2_{\mathbf w} := \sum_{j = 2}^{\infty} \frac{|\hat{f}(j)|^{2}}{w_{j}} \textless \infty\Bigg\}
\eeqn
endowed with the norm $\|\cdot\|_{\mathbf w}.$ 
If $\sigma_{\mathbf w} < +\infty,$ then $\mathscr H_{\mathbf w}$ is a reproducing kernel Hilbert space associated with $\kappa_\mathbf w$ given by
\beq \label{kappa-w}
\kappa_{\mathbf w}(s, u) = \sum_{j = 2}^{\infty} w_{j} j^{-s-\overline{u}}, \quad s, u \in \mathbb H_{\frac{\sigma_{\mathbf w}}{2}}.
\eeq
For every $s \in \mathbb H_{\frac{\sigma_{\mathbf w}}{2}},$ $\kappa_\mathbf w(\cdot, s) \in \mathscr H_\mathbf w$ and satisfies
\beq \label{rp}
\inp{f}{\kappa_\mathbf w(\cdot, s)} = f(s), \quad f \in \mathscr H_\mathbf w.
\eeq
Note that $\kappa_\mathbf w$ converges absolutely on $\mathbb H_{\frac{\sigma_{\mathbf w}}{2}} \times \mathbb H_{\frac{\sigma_{\mathbf w}}{2}}.$

The following is immediate from \cite[Theorem~3.1]{St} and the discussion prior to \cite[Theorem~2.7]{St}. 
\begin{theorem}\label{stetler_result_repeat}
	$\mathcal M(\mathscr H_\mathbf w)  \subseteq H^\infty(\mathbb H_{\frac{\delta_\mathbf w}{2}}) \cap \mathcal D$ and 
	\beqn
	\| \varphi \|_{\infty, \mathbb H_{\frac{\delta_\mathbf w}{2}}} \Le \| M_{\varphi, \mathbf w} \|_{\mathscr H_\mathbf w}, \quad \varphi \in \mathcal M(\mathscr H_\mathbf w).
	\eeqn
\end{theorem}

This raises the following question.

\begin{question}\label{mult-problem}
	For which sequences $\mathbf w,$ the multiplier algebra $\mathcal M(\mathscr H_\mathbf w)$ of $\mathscr H_{\mathbf w}$ is isometrically isomorphic to $H^\infty(\mathbb H_{\frac{\delta_\mathbf w}{2}}) \cap \mathcal D?$	
\end{question}

In the following cases, Question~\ref{mult-problem} is answered affirmatively.

\begin{itemize}
	\item [$\mbox{(i)}$]  If $\mathbf 1$ denotes the constant sequence with value $1,$ then  $\mathcal M(\mathscr H_\mathbf 1)$ is isometrically isomorphic to $H^\infty(\mathbb H_0) \cap \mathcal D$ (see \cite[Theorem~3.1]{HLS}).  
	\item [$\mbox{(ii)}$] If, for some positive Radon measure $\eta$ on $[0, \infty)$ with $0 \in \text{supp}(\eta)$ and $n_0 \in \mathbb Z_+,$ 
	\beq \label{weight-by-measure}
	\frac{1}{w_j} = \int_{0}^{\infty}j^{-2 \sigma} d\eta(\sigma), \quad j \Ge n_0,
	\eeq
	then $\mathcal M(\mathscr H_{\mathbf w_\eta})$ is isometrically isomorphic to $H^\infty(\mathbb H_0) \cap \mathcal D$ (see \cite[Theorem~1.11]{Mc-1}).
	\item [$\mbox{(iii)}$] If $\mathbf w$ is either $\{\frac{1}{d_{\beta}(n)}\}_{n=1}^\infty,$ where $\beta > 0$ and $d_{\beta}(n)$ denotes $n$-th coefficient of the Dirichlet series $\zeta^\beta,$ or $\{(\mathbf d(n))^\alpha\}_{n=1}^\infty$ for $\alpha \in \mathbb R,$
	then $\mathcal M(\mathscr H_\mathbf w)$ is isometrically isomorphic to $H^\infty(\mathbb H_0) \cap \mathcal D$ (see \cite[Section~8]{Ol}).
	\item [$\mbox{(iv)}$] If $\mathbf w = \{w_n\}_{n=1}^\infty$ is a multiplicative function (i.e., $w_{mn} = w_m w_n$ for all integers $m, n \Ge 1$ such that $\gcd(m, n) = 1$) satisfying
	\begin{align} \label{mult-growth} 
		\frac{w_{p^{j-1}}}{w_{p^j}} \Le  p^{-\delta_{\mathbf w}}~\text{for all integers} ~~ j \Ge 1~\mbox{and prime numbers}~p,
	\end{align}	 
	then $\mathcal M(\mathscr H_\mathbf w)$ is isometrically isomorphic to $H^\infty(\mathbb H_{\frac{\delta_\mathbf w}{2}}) \cap \mathcal D$ (see \cite[Corollary~4.2]{St}). 
\end{itemize}

In all the cases above, the multipliers of $\mathscr H_\mathbf w$ are extended beyond the common domain of the functions in $\mathscr H_\mathbf w$. However, this is not always true (see \cite[Theorem~3.6]{Mc-1}).

The following is the main theorem of this paper, which generalizes the aforementioned results $\mbox{(i)}$ and $\mbox{(iv)}.$
\begin{theorem}\label{extension_stetler}
	Suppose that 
	\beq \label{norm-est}
	\sum_{\substack{j \Ge 2\\ j | n}} j^{-\delta_\mathbf w} w_j \mu\Big(\frac{n}{j}\Big) \Ge 0,\quad n \Ge 2,
	\eeq
	where $\mu$ and $\delta_\mathbf w$ are given by  \eqref{mu-fun} and \eqref{delta-form}, respectively. 
	Then $\mathcal M(\mathscr H_\mathbf w)$ is isometrically isomorphic to $H^\infty(\mathbb H_{\frac{\delta_\mathbf w}{2}}) \cap \mathcal D.$ 
\end{theorem}

\begin{remark}\label{trivial-case}
	Here we make some observations:
	\begin{itemize}
		\item [(a)] If $\delta_\mathbf w = -\infty,$ then there are no non-constant multipliers of $\mathscr H_\mathbf w.$ Indeed, if $\varphi \in  \mathcal M(\mathscr H_\mathbf w),$ then by Theorem~\ref{stetler_result_repeat}, $\varphi$ is entire and bounded. Hence, by the Liouville's theorem $($see \cite[Corollary~4.5]{S-S}$),$ 
		$\varphi$ is constant.
		\item [(b)] If $\mathbf w = \{w_j\}_{j=2}^\infty$ in the above theorem is replaced by $\{w_j\}_{j=k}^\infty \subseteq (0, \infty)$ for an integer $k \in \mathbb Z_+$ and  
		\beq \label{norm-est-gen} 
		\sum_{\substack{j \Ge k\\ j | n}} j^{-\delta_\mathbf w} w_j \mu\Big(\frac{n}{j}\Big) \Ge 0,\quad n \Ge k,
		\eeq 
		then the conclusion of Theorem~\ref{extension_stetler} will be same.  
 	\end{itemize}
\end{remark}

The following consequence of Theorem~\ref{extension_stetler} deals with additive functions.
\allowdisplaybreaks
\begin{corollary}\label{addi-result}
	Let $\mathbf w$ be an additive function $($i.e., $w_{mn} = w_m + w_n$ for all integers $m, n \Ge 2$ such that $\gcd(m, n) =1)$ satisfying
	\begin{align}\label{growth-con}
		\frac{w_{p^{j-1}}}{w_{p^j}} \Le  p^{-\delta_{\mathbf w}}~\text{for all integers} ~~ j \Ge 2~\mbox{and all prime numbers}~p.
	\end{align}	 
	If $\delta_\mathbf w \Le 0,$ then $\mathcal M(\mathscr H_\mathbf w)$ is isometrically isomorphic to $H^\infty(\mathbb H_{\frac{\delta_\mathbf w}{2}}) \cap \mathcal D.$ 
\end{corollary}

The proofs of Theorem~\ref{extension_stetler} and Corollary~\ref{addi-result} will be presented in Section~\ref{S2}. 
In Section~\ref{S3}, we discuss some consequences of Theorem~\ref{extension_stetler} (see Theorem~\ref{Mul-class} and Proposition~\ref{Another-class}). 
In particular, we recover \cite[Corollary~4.2]{St}. We also present some examples of the Hilbert spaces $\mathscr H_{\mathbf w}$ illustrating the results of this paper.

\section{The multiplier algebra of $\mathscr H_\mathbf w$}\label{S2}

In view of Remark~\ref{trivial-case} (a), we assume that $\delta_\mathbf w \in \mathbb R.$ To prove Theorem~\ref{extension_stetler}, we require several lemmas.

The first one relates the condition \eqref{norm-est} with the positive semi-definiteness of a kernel.

\begin{lemma}	\label{posi-equiv}
	Let $\tilde{\mathbf w} = \{n^{-\delta_\mathbf w} w_n \}_{n=2}^\infty$ and $\beta = \frac{1}{2}\max\{\sigma_\mathbf w-\delta_\mathbf w, 1\}.$ Then the kernel $\eta:\mathbb H_\beta \times \mathbb H_\beta \rar \mathbb C$ defined by 
	\beq \label{eta-def}
	\eta (s, u) = \frac{\kappa_{\tilde{\mathbf w}}(s, u)}{\kappa_{\mathbf 1}(s, u)}, \quad s, u \in \mathbb H_\beta
	\eeq
	is positive semi-definite if and only if \eqref{norm-est} holds.
\end{lemma}
\begin{proof} By \cite[Example~$11.4.1$]{Ap}, 
	\beqn
	\frac{1}{\zeta(s)} = \sum_{j=1}^\infty \mu(j) j^{-s}, \quad s \in \mathbb H_1.
	\eeqn
	This, combined with   
	Proposition~\ref{prod-series} and \eqref{kappa-w}, yields 
	\beqn 
	\eta(s, u) &=& \Big(\sum_{j= 2}^\infty j^{-\delta_\mathbf w}w_j j^{-s-\overline{u}}\Big)  \Big(\sum_{j=1}^\infty \mu(j) j^{-s-\overline{u}}\Big) \notag\\
	&=& \sum_{n = 2}^\infty \Big(\sum_{\substack{j \Ge 2 \\ j | n}} j^{-\delta_\mathbf w}w_j \mu\Big(\frac{n}{j}\Big)\Big) n^{-s-\overline{u}}, \quad s, u \in \mathbb H_{\beta}.
	\eeqn
	The desired equivalence is now immediate from \cite[Lemma~20]{MS}. 
\end{proof}

The following lemma is needed in the proof of Theorem~\ref{extension_stetler}.

\begin{lemma}\label{kernel_shifting} 
	For $t \in \mathbb R,$ let $\tilde{\mathbf w} = \{n^{-t}w_n\}_{n= 2}^\infty.$ 
	If $\varphi \in \mathcal M(\mathscr H_\mathbf w),$ then $\tilde{\varphi} : \mathbb H_{\frac{\sigma_\mathbf w - t}{2}} \rar \mathbb C$ defined by 
	\beqn \tilde{\varphi}(s) = \varphi\Big(s + \frac{t}{2}\Big), \quad s \in \mathbb H_{\frac{\sigma_\mathbf w - t}{2}}
	\eeqn
	is a multiplier of $\mathscr H_{\tilde{\mathbf w}}$ $($see \eqref{kappa-w}$).$
\end{lemma}
\begin{proof}
	Let $j: \mathscr H_\mathbf w \rar \mathscr H_{\tilde{\mathbf w}}$ be a map defined by 
	\beqn 	
	j(f)(s) = f\Big(s + \frac{t}{2}\Big), \quad s \in \mathbb H_{\frac{\sigma_\mathbf w - t}{2}}. 
	\eeqn	
	Since $\varphi \in \mathcal M(\mathscr H_\mathbf w),$ for any $f \in \mathscr H_{\mathbf w},$
	\beqn
	\tilde{\varphi}(s)j(f)(s) &=& \varphi \Big(s + \frac{t}{2}\Big) f \Big(s + \frac{t}{2}\Big)\\
	&=& (\varphi \cdot f) \Big(s + \frac{t}{2}\Big) = j(\varphi \cdot f)(s),~~ s \in \mathbb H_{\frac{\sigma_\mathbf w - t}{2}}.
	\eeqn
	As $j$ is surjective,  $\tilde{\varphi}\cdot g \in \mathscr H_{\tilde{\mathbf w}}$ for every $g \in \mathscr H_{\tilde{\mathbf w}}.$ This concludes the proof.
\end{proof}

The following fact shows that the inclusion map $\mathcal M(\mathscr H_\mathbf w) \xhookrightarrow{} H^\infty(\mathbb H_{\frac{\delta_\mathbf w}{2}}) \cap \mathcal D$ is surjective (cf. \cite[Theorem~4.5]{St}).

\begin{lemma}\label{Imp_theorem1}
	If \eqref{norm-est} holds, then 
	$H^\infty(\mathbb H_{\frac{\delta_\mathbf w}{2}}) \cap \mathcal D = \mathcal M(\mathscr H_\mathbf w)$ $($as a set$).$
\end{lemma}

\begin{proof}
	Assume that \eqref{norm-est} holds. In view of Theorem~\ref{stetler_result_repeat}, it suffices to check that 
	$H^\infty(\mathbb H_{\frac{\delta_\mathbf w}{2}}) \cap \mathcal D \subseteq \mathcal M(\mathscr H_\mathbf w).$
	To see that, let $\varphi \in H^\infty(\mathbb H_{\frac{\delta_\mathbf w}{2}}) \cap \mathcal D.$ Then $\tilde{\varphi}: \mathbb H_0 \rar \mathbb C$ defined by $\tilde{\varphi}(s) = \varphi (s+\frac{\delta_\mathbf w}{2}),~ s \in \mathbb H_0,$ belongs to $H^\infty(\mathbb H_0) \cap \mathcal D.$ Hence, by \cite[Theorem~3.1]{HLS}, 
	$\tilde{\varphi}$ is a multiplier of $\mathscr H_\mathbf 1.$ Therefore, by \cite[Theorem~5.21]{P-R}, there exists $c \Ge 0$ such that 
	\beqn
	(c^2 - \tilde{\varphi} (s) \overline{\tilde{\varphi}(u)}) \kappa_\mathbf 1(s, u) \Ge 0.
	\eeqn
	Let $\tilde{\mathbf w} = \{ n^{-\delta_\mathbf w}w_n\}_{n=2}^\infty.$ Then, by Lemma~\ref{posi-equiv} and \cite[Theorem~4.8]{P-R}, 
	\beqn
	(c^2 - \tilde{\varphi} (s) \overline{\tilde{\varphi}(u)}) \kappa_{\tilde{\mathbf w}}(s, u) \Ge 0.
	\eeqn
	Thus, by \cite[Theorem~5.21]{P-R}, $\tilde{\varphi} \in \mathcal M(\mathscr H_{\tilde{\mathbf w}}).$ An application of Lemma~\ref{kernel_shifting} now completes the proof.
\end{proof}

\begin{proof} [{\it Proof of Theorem~\ref{extension_stetler}}]
	By Lemma~\ref{Imp_theorem1}, 
	\beq
	\label{appli-lemma-Imp}
	\mbox{$H^\infty(\mathbb H_{\frac{\delta_\mathbf w}{2}}) \cap \mathcal D = \mathcal M(\mathscr H_\mathbf w)$ $($as a set$).$}
	\eeq
	Hence, in view of Theorem~\ref{stetler_result_repeat}, it is sufficient to check that
	\beqn
	\|M_{\varphi, \mathbf w}\|_{\mathscr H_\mathbf w} \Le \|\varphi\|_{\infty, \mathbb H_{\frac{\delta_\mathbf w}{2}}}, \quad \varphi \in H^\infty(\mathbb H_{\frac{\delta_\mathbf w}{2}}) \cap \mathcal D.
	\eeqn
	
	To see the above estimate, let $\varphi \in H^\infty(\mathbb H_{\frac{\delta_\mathbf w}{2}}) \cap \mathcal D.$ Note that $\tilde{\varphi} :\mathbb H_0 \rar \mathbb C,$ defined by $\tilde{\varphi}(s) = \varphi(s + \frac{\delta_\mathbf w}{2}),~s \in \mathbb H_0,$ belongs to $H^\infty(\mathbb H_0) \cap \mathcal D.$ Therefore, by an application of \cite[Theorem~3.1]{HLS}, 
	\beqn 
	\mbox{the multiplication operator $M_{\tilde{\varphi}, \mathbf 1}$ is bounded on $\mathscr H_\mathbf 1.$}
	\eeqn
	In particular,  the Hilbert space adjoint $M^*_{\tilde{\varphi}, \mathbf 1}$ of $M_{\tilde{\varphi}, \mathbf 1}$ is well-defined as a bounded linear operator. 
	
	Note that Lemma~\ref{posi-equiv} combined with \eqref{norm-est} yields that the kernel $\eta$ defined by \eqref{eta-def}, is positive semi-definite on $\mathbb H_\beta \times \mathbb H_\beta,$ where $\beta = \frac{1}{2}\max\{\sigma_\mathbf w-\delta_\mathbf w, 1\}.$ Hence, by \cite[Theorem~2.14]{P-R}, there exists a reproducing kernel Hilbert space $\mathcal H(\eta)$ associated with $\eta.$ Thus, by \eqref{rp}, for any $s, u \in \mathbb H_\beta,$
	\beqn
	\inp{\kappa_{\tilde{\mathbf w}}(\cdot, u)}{\kappa_{\tilde{\mathbf w}}(\cdot, s)}_{\mathscr H_{\tilde{\mathbf w}}} = \inp{\kappa_\mathbf 1(\cdot, u) \otimes \eta(\cdot, u)}{\kappa_\mathbf 1(\cdot, s) \otimes \eta(\cdot, s)}_{\mathscr H_\mathbf 1 \otimes \mathcal H(\eta)}.
	\eeqn
	An application of the Lurking isometry Lemma (see \cite[Lemma~2.18]{AMY}) yields a linear isometry $V: \mathscr H_{\tilde{\mathbf w}} \rar \mathscr H_\mathbf 1 \otimes \mathcal H(\eta)$ such that
	\beq \label{isometry}
	V(\kappa_{\tilde{\mathbf w}}(\cdot, u)) = \kappa_\mathbf 1(\cdot, u) \otimes \eta(\cdot, u), \quad u \in \mathbb H_{\beta}.
	\eeq
	
	Further, since $\varphi \in \mathcal M(\mathscr H_\mathbf w)$ (see \eqref{appli-lemma-Imp}), by Lemma~\ref{kernel_shifting}, $\tilde{\varphi } \in \mathcal M(\mathscr H_{\tilde{\mathbf w}}).$ This implies that the multiplication operator $M_{\tilde{\varphi}, {\tilde{\mathbf w}}}$ is bounded on $\mathscr H_{\tilde{\mathbf w}},$ and hence, by \cite[Corollary~5.22]{P-R},
	\beqn
	V^{*} (M_{\tilde{\varphi}, \mathbf 1}^* \otimes I)V (\kappa_{\tilde{\mathbf w}}(\cdot, u)) &\overset{ \eqref{isometry}}=& V^{*} (M_{\tilde{\varphi}, \mathbf 1}^* \otimes I)( \kappa_\mathbf 1(\cdot, u) \otimes \eta(\cdot, u))\\ &=& \overline{\tilde{\varphi}(u)}V^* ( \kappa_\mathbf 1(\cdot, u) \otimes \eta(\cdot, u))\\
	&=& M_{\tilde{\varphi}, {\tilde{\mathbf w}}}^* (\kappa_{\tilde{\mathbf w}}(\cdot, u)),\quad u \in \mathbb H_\beta.
	\eeqn
	Since $V^{*} (M_{\tilde{\varphi}, \mathbf 1}^* \otimes I)V$ and  $M_{\tilde{\varphi}, {\tilde{\mathbf w}}}^*$ are bounded linear operators on $\mathscr H_{\tilde{\mathbf w}},$  
	\beqn
	V^{*} (M_{\tilde{\varphi}, \bf 1}^* \otimes I)V= M_{\tilde{\varphi}, \tilde{\mathbf w}}^*.
	\eeqn
	Because $V$ is an isometry, we obtain
	\beqn
	\|M_{\tilde{\varphi}, \tilde{\mathbf w}}\|_{_{_{\mathscr H_{\tilde{\mathbf w}}}}} \Le \|V^{*} \|  \|M_{\tilde{\varphi}, \bf 1}^* \|  \| V\| \Le \|\tilde{\varphi}\|_{\infty,\, \mathbb H_0} = \|\varphi\|_{\infty,\, \mathbb H_\frac{\delta_\mathbf w}{2}},
	\eeqn
	which completes the proof.
\end{proof}
\begin{remark}
	Let $\kappa$ be given by
	\beqn
	\kappa(z, w) = \sum_{j=0}^\infty a_j z^j \overline{w}^j.
	\eeqn 
	Suppose that $\kappa$ converges on $\mathbb D \times \mathbb D.$ If $\mathbf a$ is an increasing sequence of non-negative real numbers, 
	then the multiplier algebra $\mathcal M(\mathscr H(\kappa))$ of $\mathscr H(\kappa)$ is isometrically isomorphic to $H^\infty(\mathbb D).$ This can be shown using the arguments in the proof of Theorem~\ref{extension_stetler} with the essential change that  $\kappa_\mathbf 1$ is replaced by the Szeg$\ddot{\mbox{o}}$ kernel $S: \mathbb D \times \mathbb D \rar \mathbb C$ defined by
	\beqn
	S(z, w) = \frac{1}{1-z\overline{w}}, \quad z, w \in \mathbb D.
	\eeqn
\end{remark}

\begin{proof} [{\it Proof of Corollary~\ref{addi-result}}]
	Fix an integer $n \Ge 2,$ let $\prod\limits_{m=1}^{\omega(n)} p_{i_m}^{r_m},~r_m \in \mathbb Z_+,$ be the prime factorization of $n.$ If $j \in \mathbb Z_+$ such that $j | n,$ then
	\beq \label{factor-of-n}
	j = \prod_{m=1}^{\omega(n)} p_{i_m}^{s_{m, j}},\quad  (s_{m, j})_{m=1}^{\omega(n)} \in \prod_{m=1}^{\omega(n)} \mathbb N_{r_m}.
	\eeq
	Hence, there exists a bijective map $\psi : \!\{j \in \mathbb Z_+ \! :j | n,~ j \Ge 2\} \rar \prod\limits_{m=1}^{\omega(n)} \mathbb N_{r_m} - \{\bf 0\}$ defined by $\psi (j) = (s_{m, j})_{m=1}^{\omega(n)}.$ Therefore,
	by \eqref{factor-of-n},
	
	\beq 
	&&	\notag \sum_{\substack{j \Ge 2 \\ j | n}} j^{-\delta_\mathbf w} w_j \mu\Big(\frac{n}{j}\Big) \\ \label{exp-M}
	&=& \sum_{(s_m)_{m=1}^{\omega(n)} \in \, \text{Im}(\psi)}\!\!\! \Big(\prod_{m=1}^{\omega(n)} p_{i_m}^{-\delta_\mathbf w s_m}\Big) \big(w_{_{\prod_{m=1}^{\omega(n)} p_{i_m}^{s_m}}}\big) \mu\Big({\prod_{m=1}^{\omega(n)} p_{i_m}^{r_m-s_m}}\Big).
	\eeq
	Let $\tilde{\mathbf w}$ be the extension of $\mathbf w$ to $\mathbb Z_+$ by letting $w_1=0.$ Since $\mathbf w$ is additive, $\tilde{\mathbf w}$ is additive.  
	This together with \eqref{exp-M} and the multiplicativity of $\mu$ gives 
	\allowdisplaybreaks
	\beqn
	\sum_{\substack{j \Ge 2 \\ j | n}} j^{-\delta_\mathbf w} w_j \mu\Big(\frac{n}{j}\Big)\!\!\!\!\!\!
	&=&\!\!\!\!\!\! \sum_{(s_m)_{m=1}^{\omega(n)} \in \,\text{Im}(\psi)} \sum_{t=1}^{\omega(n)} \Bigg(w_{p_{i_t}^{s_t}} \Big(\prod_{m=1}^{\omega(n)} p_{i_m}^{-\delta_\mathbf w s_m}\Big) \mu\Big({\prod_{m=1}^{\omega(n)} p_{i_m}^{r_m-s_m}}\Big) \Bigg) \notag \\
	&=& \sum_{t=1}^{\omega(n)}\! \sum_{(s_m)_{m=1}^{\omega(n)} \in \,\text{Im}(\psi)} \!\!\!\!\!\! \Big(w_{p_{i_t}^{s_t}} \Big(\!\!\prod_{m=1}^{\omega(n)} p_{i_m}^{-\delta_\mathbf w s_m}\Big) \Big(\prod_{m=1}^{\omega(n)} \mu(p_{i_m}^{r_m-s_m})\Big)\!\Big)\\
	&\overset{\eqref{mu-fun}}{=}& \sum_{t=1}^{\omega(n)} \sum_{(s_m)_{m=1}^{\omega(n)} \in \, \mathcal U_n} \!\!\!\!\!\! \Big(w_{p_{i_t}^{s_t}} \prod_{m=1}^{\omega(n)} \Big(p_{i_m}^{-\delta_\mathbf w s_m} \mu(p_{i_m}^{r_m-s_m})\Big)\Big),
	\eeqn 
	where $\mathcal U_n = \prod\limits_{m=1}^{\omega(n)} \{r_m-1, r_m\}.$ For any positive integer $t \Le \omega(n),$ let
	\beqn
	T_t =  \sum_{(s_m)_{m=1}^{\omega(n)} \in \,\mathcal U_n} \!\!\!\!\!\! \Big(w_{p_{i_t}^{s_t}} \prod_{m=1}^{\omega(n)} \Big(p_{i_m}^{-\delta_\mathbf w s_m} \mu(p_{i_m}^{r_m-s_m})\Big)\Big).
	\eeqn
	Then 
	\beqn
	T_t\!\!
	&\!\!=& \sum_{s_1 = \,r_1-1}^{r_1} \sum_{s_2 = \,r_2-1}^{r_2} \ldots \sum_{s_{\omega(n)} = \,r_{\omega(n)}-1}^{r_{\omega(n)}} 
	\Big(w_{p_{i_t}^{s_t}} \prod_{m=1}^{\omega(n)} \Big(p_{i_m}^{-\delta_\mathbf w s_m} \mu(p_{i_m}^{r_m-s_m})\Big)\Big)\\
	&\!\!=& \!\!\Big(\!\!\sum_{s_t = r_t-1}^{r_t} \!\!\! p_{i_t}^{-\delta_{_\mathbf w} s_t}  w_{p_{i_t}^{s_t}} \mu(p_{i_t}^{r_t-s_t})\Big) \prod_{\substack{m=1 \\ m \neq t}}^{\omega(n)} \Big(\sum_{s_m = r_m-1}^{r_m}  p_{i_m}^{-\delta_\mathbf w s_m} \mu(p_{i_m}^{r_m-s_m})\Big)\\
	&=&  \big( p_{i_t}^{-\delta_{_\mathbf w}(r_t-1)} w_{p_{i_t}^{r_t-1}} \mu(p_{i_t}) +  p_{i_t}^{-\delta_{\mathbf w}r_t} w_{p_{i_t}^{r_t}} \mu(1)\big)\\
	&& \quad \quad \quad \quad \quad \quad \quad \quad \quad \quad \prod_{\substack{m=1 \\ m \neq t}}^{\omega(n)} \big(p_{i_m}^{-\delta_{_\mathbf w}(r_m-1)} \mu(p_{i_m}) +  p_{i_m}^{-\delta_{\mathbf w}r_m}\mu(1)\big)\\
	&\!\!\overset{\eqref{mu-fun}}{=}& p_{i_t}^{-\delta_\mathbf w(r_t-1)}(p_{i_t}^{-\delta_\mathbf w}w_{p_{i_t}^{r_t}}-w_{p_{i_t}^{r_t-1}})
	\prod_{\substack{m=1 \\ m \neq t}}^{\omega(n)} \big(p_{i_m}^{-\delta_\mathbf w(r_m-1)} (p_{i_m}^{-\delta_\mathbf w}-1)\big).
	\eeqn
	This combined with the assumptions $\delta_\mathbf w \Le 0$ and  \eqref{growth-con} (which imply that $p_{i_m}^{-\delta_\mathbf w} \Ge 1$ and $p_{i_t}^{-\delta_\mathbf w}w_{p_{i_t}^{r_t}} \Ge w_{p_{i_t}^{r_t-1}}$) shows that $T_t \Ge 0$ for every positive integer $t \Le \omega(n).$ 
	Since $n$ is arbitrary, \eqref{norm-est} is valid for all integers $n \Ge 2.$ An application of Theorem~\ref{extension_stetler} now completes the proof. 
\end{proof}

We provide an example illustrating Corollary~\ref{addi-result}.

\begin{example}
	Let $\mathbf w = \omega.$ Note that $\omega$ is additive and satisfies
	\beq
	\label{imp-pro}
	\text{$\omega(p^j) = 1$ for every prime $p$ and integer $j \Ge 1$}.
	\eeq
	By \cite[Chap.~I, Equations~(1.6.1) and (1.6.2)]{Ti}, 
	\beq \label{series-expan}
	\sum_{j=2}^\infty \omega(j) j^{-s} = \zeta(s) \sum_{j = 1}^\infty p_j^{-s},\quad s \in \mathbb H_1.
	\eeq
	It follows from \eqref{imp-pro}, \eqref{series-expan} and \cite[Theorem~1.13]{Ap} that $\sigma_\omega = 1.$ 
	To compute $\delta_\omega,$ consider any $\epsilon > 0$ and an integer $n \Ge 1,$ 
	\beq \label{trunc-series}
	\sum_{\tiny{\mbox{\bf gpf}}(j) \Le p_n} \omega(j) j^{-\epsilon} \!\!&=&  \sum_{j=1}^n \sum_{\substack{i_1, i_2, \ldots, i_j =1 \\ i_s \neq i_t,~ s\neq t}}^n  \Bigg(\sum_{m_{i_1}, \ldots, m_{i_j} = 1}^\infty  j(p_{i_1}^{m_{i_1}} p_{i_2}^{m_{i_2}} \ldots p_{i_j}^{m_{i_j}})^{-\epsilon}\Bigg) \notag\\
	&=&  \sum_{j=1}^n \Bigg(\sum_{\substack{i_1, i_2, \ldots, i_j =1 \\ i_s \neq i_t,~ s\neq t}}^n \Big(j \prod_{r=1}^j \sum_{m_{i_r} = 1}^\infty  p_{i_r}^{-\epsilon m_{i_r}}\Big)\Bigg) < \infty.
	\eeq
	In view of \eqref{imp-pro}, $\sum_{\tiny{\mbox{\bf gpf}}(j) \Le p_n} \omega(j)$ diverges for all $n \Ge 1.$ This combined with \eqref{trunc-series} yields $\delta_\omega = 0,$ and hence, \eqref{growth-con} holds. 
	Therefore, by Corollary~\ref{addi-result}, $\mathcal M(\mathscr H_\omega)$ is isometrically isomorphic to $H^\infty(\mathbb H_0) \cap \mathcal D.$ 
	Moreover,
	\beqn
	\omega(pq) = 2 \overset{\eqref{imp-pro}}{\neq} \omega(p)\omega(q) \ \text{for distinct primes $p$ and $q.$}
	\eeqn
	So, $\omega$ is not multiplicative. Also, because $\omega$ is not monotone, \eqref{weight-by-measure}  does not hold. 
\end{example}

\section{Applications}\label{S3}
In this section, we discuss some applications of Theorem~\ref{extension_stetler}. 
The first one recovers \cite[Corollary~4.2]{St}. 

\begin{theorem}\label{Mul-class}
	Let $\mathbf w = \{w_j\}_{j=1}^\infty$ be an multiplicative function satisfying \eqref{mult-growth}.
	Then $\mathcal M(\mathscr H_\mathbf w)$ is isometrically isomorphic to $H^\infty(\mathbb H_{\frac{\delta_\mathbf w}{2}}) \cap \mathcal D.$
\end{theorem}
\begin{proof}
	Fix an integer $n \Ge 2,$ let $\prod_{m=1}^{\omega(n)} p_{i_m}^{r_m},~r_m \in \mathbb Z_+,$ be the prime factorization of $n.$ Since any divisor of $n$ is given by $\prod_{m=1}^{\omega(n)} p_{i_m}^{s_m},~ s_m \in \mathbb N_{r_m},$ 
	\beqn
	\sum_{\substack{j \Ge 1\\ j | n}} j^{-\delta_\mathbf w} w_j \mu\Big(\frac{n}{j}\Big) &=& \prod_{m=1}^{\omega(n)} \sum_{s_m = 0}^{r_m} p_{i_m}^{-\delta_\mathbf w s_m} \, w_{p^{s_m}_{i_m}} \, \mu (p_{i_m}^{r_m - s_m})\\
	&\overset{\eqref{mu-fun}}=&  \prod_{m=1}^{\omega(n)} (p_{i_m}^{-\delta_\mathbf w(r_m-1)}) (p_{i_m}^{-\delta_\mathbf w} w_{p_{i_m}^{r_m}} - w_{p_{i_m}^{r_m-1}}) \overset{\eqref{mult-growth}}\Ge 0.
	\eeqn
	Thus, \eqref{norm-est-gen} is valid for all integers $n \Ge 2$ and is trivially true for $n = 1.$ Hence by Theorem~\ref{extension_stetler} and Remark~\ref{trivial-case}(b), the conclusion now follows.
\end{proof}

As a consequence of the previous theorem, we recover \cite[Theorem~3]{BB}. Indeed, 
for any $\alpha > 0,$ consider $\mathbf w_\alpha = \{(\mathbf d(j))^\alpha\}_{j=1}^\infty.$ Since $\mathbf d(n) = o(n^\epsilon)$ for every $\epsilon > 0$ (see \cite[Equation~(31)]{Ap}), $\sigma_{\mathbf w_\alpha} = 1$ and $\delta_{\mathbf w_\alpha} = 0.$ For any prime $p$ and integer $j \in \mathbb Z_+, \mathbf d(p^j) = j+1.$ Thus, \eqref{growth-con} holds, and hence by Theorem~\ref{Mul-class}, $\mathcal M(\mathscr H_{\mathbf w_{\alpha}})$ is isometrically isomorphic to $H^\infty(\mathbb H_0) \cap \mathcal D.$

The following proposition yields a family of Hilbert spaces $\mathscr H_\mathbf w$ illustrating the main result.

\begin{proposition}\label{Another-class}
	Let $\mathbf w = \{w_j\}_{j=1}^\infty$ be a multiplicative function satisfying 
	\begin{itemize}
		\item [$(i)$] $w_j = \mathcal{O}(j^{\delta})$ for every $\delta > 0,$
		\item [$(ii)$] for each prime $p,$ $\{w_{p^{j}}\}_{j=0}^\infty$ is increasing. 
	\end{itemize}   
	If $\tilde{\mathbf w} = \{1+ w_j\}_{j=1}^\infty,$ then $\mathcal M(\mathscr H_{\tilde{\mathbf w}})$ is isometrically isomorphic to $H^\infty(\mathbb H_0) \cap \mathcal D.$
\end{proposition}
\begin{proof} 
	For a real number $r > 1,$ let $\delta$ be a positive number such that $r > \delta + 1.$ Then by $(i),$
	\beq
	\label{value-sigma}
	\sum_{j=2}^\infty (1+w_j) j^{-r} \Le \sum_{j=2}^\infty j^{-r} + C_\delta \sum_{j=2}^\infty j^{-(r-\delta)} < \infty,
	\eeq
	which yields $\sigma_{\tilde{\mathbf w}} \Le 1.$ Moreover, since $\sum_{j=2}^\infty j^{-1}$ diverges, $\sigma_{\tilde{\mathbf w}} = 1.$ By using
	$(i)$, we also get that $\delta_{\tilde{\mathbf w}} = 0.$ For any $\sigma > 0,$ let $\delta$ be such that $0 < \delta < \sigma.$ Then, \cite[Lemma~3.2]{St} combined with $(i)$ yields
	\beq
	\label{trun-con}
	\sum_{\substack{j \Ge 2 \\\tiny{\mbox{\bf gpf}}(j) \Le p_n }} w_j j^{- \sigma} \Le \sum_{\substack{j \Ge 2 \\\tiny{\mbox{\bf gpf}}(j) \Le p_n }} j^{-(\sigma -\delta)} < \infty, \quad n \Ge 1.
	\eeq
	In view of $(ii),$ the series $\sum_{j=0}^\infty w_{2^j}$ diverges. This together with \eqref{trun-con} yields $\delta_\mathbf w = 0.$ Hence, by $(ii),$ \eqref{growth-con} holds. 
	Thus, by using the steps in the proof of Theorem~\ref{Mul-class}, we obtain
	\beqn
	\sum_{\substack{j \Ge 1 \\ j | n}} w_j \mu\Big(\frac{n}{j}\Big) \Ge 0, \quad n \Ge 1.
	\eeqn
	Combining this with \cite[Theorem~2.1]{Ap} gives
	\beqn 
	\sum_{\substack{j \Ge 1 \\ j | n}} (1+w_j) \mu\Big(\frac{n}{j}\Big) \Ge 0, \quad n \Ge 1.
	\eeqn
	The desired conclusion is now immediate from Theorem~\ref{extension_stetler} together with Remark~\ref{trivial-case}(b). 
\end{proof}

Note that $\tilde{\mathbf w}$ defined above is never multiplicative. Moreover, Proposition~\ref{Another-class} applies to $\mathbf w = \{\mathbf d(j)\}_{j=1}^\infty.$ We conclude this paper by revealing that Theorem~\ref{extension_stetler} also recovers \cite[Theorem~1.11]{Mc-1} for some particular sequences. 

\begin{example}
	Let $\mathbf w$ be given by \eqref{weight-by-measure}. Then by \cite[Equation~(1.4)]{Mc-1}, $\sigma_\mathbf w \Le 1.$ Since $\mathbf w$ is an increasing sequence, $\sigma_\mathbf w =  1.$ So, $\kappa_\mathbf w$ converges absolutely on $\mathbb H_{1/2} \times \mathbb H_{1/2}.$ For an integer $\alpha > 0,$ consider the sequence $\mathbf w_\alpha = \{(\log(j))^\alpha\}_{j=2}^\infty.$ By \cite[Section~1]{Mc-1}, 
	\beqn
	(\log(j))^{-\alpha} = \int_{[0, \infty)} j^{-2 \sigma} \frac{2^\alpha}{\Gamma(\alpha)} \sigma^{-1+\alpha} d\sigma,\quad j \geq 2,
	\eeqn
	where $\Gamma$ denotes the Gamma function, and hence $\sigma_{\mathbf w_\alpha} = 1.$ Also, combining \cite[Equation~(1.4)]{Mc-1} with \cite[Lemma~2]{St} yields $\delta_{\mathbf w_\alpha} = 0.$ In addition, by \cite[Equation~(3.16)]{F-I}, we obtain the non-negativity of the {\it Generalized von Mangoldt function}:
	\beq 
	\label{subclass-non}
	\sum_{\substack{j \Ge 2\\ j | n}} (\log(j))^{\alpha} \mu\Big(\frac{n}{j}\Big) \Ge 0.
	\eeq
	An application of Theorem~\ref{extension_stetler} now shows that $\mathcal M(\mathscr H_{\mathbf w_\alpha})$ is isometrically isomorphic to $H^\infty(\mathbb H_0) \cap \mathcal D.$ 
\end{example}

The previous example yields that for any integer $\alpha > 0,$  $\mathbf w_\alpha$ satisfies \eqref{subclass-non}. In general, we don't know whether $\mathbf w$ satisfying \eqref{weight-by-measure} guarantees that  
\beqn
\sum_{\substack{j \Ge n_0 \\ j | n}} w_{j} \mu\Big(\frac{n}{j}\Big) \Ge 0,\quad n \Ge n_0.
\eeqn

\medskip \textit{Acknowledgment}. \
I am grateful to Sameer Chavan for suggesting this topic. Needless to say, this paper would not have seen the present form without his continuous support.   
I am also thankful to the anonymous referee for several valuable inputs improving the presentation of the paper.

{}

\end{document}